\documentclass{elsarticle}

\usepackage[utf8]{inputenc}

\usepackage{amssymb}
\usepackage{amsmath}
\usepackage{amsthm}

\usepackage{pgf,tikz}
\usetikzlibrary{arrows}

\numberwithin{equation}{section}
\theoremstyle{plain}
\newtheorem{thm}{Theorem}[section]
\newtheorem{lemma}[thm]{Lemma}
\newtheorem{prop}[thm]{Proposition}
\newtheorem{coroll}[thm]{Corollary}

\theoremstyle{definition}
\newtheorem{definition}[thm]{Definition}

\newtheorem{remark}[thm]{Remark}
\newtheorem*{notation}{Notation}

\usepackage{xspace} 

\newcommand{\bigslant}[2]{{\raisebox{.2em}{$#1$}\left/\raisebox{-.2em}{$#2$}\right.}}

\def\Pic{{\rm Pic}}
\def\Div{{\rm Div}}

\def\deg{{\rm deg}}

\def\Arb{{\rm Arb}}
\def\Im{{\rm Im}}
\def\per{{\rm per}}

\makeatletter
\def\ps@pprintTitle{%
	\let\@oddhead\@empty
	\let\@evenhead\@empty
	\def\@oddfoot{\reset@font\hfil\thepage\hfil}
	\let\@evenfoot\@oddfoot
}
\makeatother

\begin{document}

\title{Algorithmic aspects of rotor-routing and the notion of linear equivalence}

\author{Lilla T\'othm\'er\'esz}
\ead{tmlilla@cs.elte.hu}

\address{Department of Computer Science, E\"otv\"os Lor\'and University, P\'azm\'any P\'eter s\'et\'any 1/C, Budapest H-1117, Hungary}

\date{}

\begin{abstract}
We define the analogue of linear equivalence of graph divisors for the rotor-router model, and use it to prove polynomial time computability of some problems related to rotor-routing.
Using the connection between linear equivalence for chip-firing and for rotor-routing, we give a simple proof for the fact that the number of rotor-router unicycle-orbits equals the order of the Picard group. We also show that the rotor-router action of the Picard group on the set of spanning in-arborescences can be interpreted in terms of the linear equivalence.
\end{abstract}


\maketitle

\section{Introduction} \label{sec::intro}

Rotor-routing is a deterministic process that induces a walk of a chip on a directed graph.
It was introduced in the physics literature as a model of self-organized criticality \cite{Priez96,Rabani98,twoballs}.
The rotor walk can also be thought of as a derandomized random walk on a graph \cite{rr_and_markov}.

In this paper, we explore the relationship of rotor-routing with the chip-firing game, and the Picard group of the graph.
We analyze a generalized version of rotor-routing, where each vertex has an integer number of chips, which might also be negative. This model has sometimes been called the height-arrow model \cite{height-arrow04}. Rotor-routing in this setting becomes a one-player game analogous to chip-firing, where a vertex can make a step if it has a positive number of chips.

In Section \ref{sec::rec_char}, we characterize recurrent elements for the rotor-routing game. This result is a generalization of a result of Holroyd et al.~\cite{Holroyd08} that characterizes recurrent configurations with one chip.
A motivation for such a characterization is the fact that for the chip-firing game, no characterization is known for the recurrent elements on general digraphs.

In Section \ref{sec::lin_eqv}, we define the analogue of the notion of linear equivalence of the chip-firing game for the rotor-routing game.
We show that the linear equivalence notions of the two models are related in a simple way. Moreover, whether two configurations of the rotor-routing game are linearly equivalent can be decided in polynomial time.

We use this result to prove polynomial time decidability of the reachability problem for rotor-routing in a special case.  
In particular, we show, that it can be decided in polynomial time whether two unicycles lie in the same rotor-router orbit.
Using the relationship between linear equivalence for chip-firing and for rotor-routing, we give a simple bijective proof for the fact that the number of rotor-router unicycle orbits equals the order of the Picard group of the graph. (This fact also follows from a combination of previous results \cite[Theorem 1]{pham_rotor} and \cite[Theorem 2.10]{farrell-levine-coeulerian}, but they do not provide a bijection.) 
Finally, we show, that the rotor-router action of the Picard group on the set of spanning in-arborescences \cite{Holroyd08}
can also be interpreted in terms of the linear equivalence.
Using this interpretation, we show that
it can be checked in polynomial time, whether a given spanning in-arborescence is the image of another given arborescence by a given element of the Picard group.
Also using this interpretation, we give a simpler proof for the result of Chan et al.~\cite{Chan15} stating that the rotor-router action is independent of the base point if and only if all cycles in the graph are reversible.

\subsection{Basic notations}

Throughout this paper, \emph{digraph} means a directed graph, where multiple edges are allowed, but there are no loops. We will almost always assume our digraphs to be strongly connected.
For a digraph $G$, $V(G)$ denotes the set of vertices, and $E(G)$ denotes the set of edges.
For a directed edge $\overrightarrow{uv}$, $u$ is the \emph{tail}, and $v$ is the \emph{head}. The multiplicity of the edge $\overrightarrow{uv}$ is denoted by $d(u,v)$.
We denote the set of out-neighbors (in-neighbors) of a vertex $v$ by $\Gamma^+(v)$ ($\Gamma^-(v)$), the out-degree (in-degree) of a vertex $v$ by $d^+(v)$ ($d^-(v)$).

  For a digraph $G$ and vertex $w\in V(G)$ a \emph{spanning in-arbores\-cence} of $G$ rooted at $w$ is a subdigraph $G'$ such that
$d^+_{G'}(v)=1$ for each $v\in V(G)-w$, and the underlying undirected graph of $G'$ is a tree.

We denote by $\mathbb{Z}^{V(G)}$ the set of integer vectors indexed by the vertices of a digraph $G$.
We identify vectors in $\mathbb{Z}^{V(G)}$ with integer valued functions on $V(G)$. According to this, we  write $z(v)$ for the coordinate corresponding to vertex $v$ of a $z\in \mathbb{Z}^{V(G)}$.
We denote by $z\geq 0$ if a vector $z\in \mathbb{Z}^{V(G)}$ is coordinatewise nonnegative.
We use the notation $\mathbf{0}_G$ ($\mathbf{1}_G$) for the vector where each coordinate equals zero (one). We denote the characteristic vector of a vertex $v$ by $\mathbf{1}_v$.

\begin{definition}
The \emph{Laplacian matrix} of a digraph $G$ is the
following matrix $L_G\in \mathbb{Z}^{V(G)\times V(G)}$:
$$L_G(u,v) = \left\{\begin{array}{cl} -d^+(u) & \text{if } u=v, \\
        d(v, u) & \text{if } u\neq v.
      \end{array} \right.$$
\end{definition}

\begin{prop}\cite[Proposition 4.1 and 3.1]{BL92}
 For a strongly connected digraph $G$, there exists a unique vector $\per_G\in \mathbb{Z}^{V(G)}$ such that $L_G\per_G=\mathbf{0}_G$, the entries of $\per_G$ are strictly positive, and relatively prime. If $G$ is Eulerian, then $\per_G=\mathbf{1}_G$.
\end{prop}
The vector $\per_G$ is called the primitive period vector of $G$.

\subsection{Chip-firing}

Chip-firing is a solitary game on a directed graph. 
The configurations of the game are called \emph{divisors}. A divisor $x$ is an integer vector indexed by 
the vertices of the graph, i.e.~$x\in \mathbb{Z}^{V(G)}$.
We think of $x(v)$ as the number of chips on vertex $v$ (which might be negative).
The \emph{degree} of a divisor is the sum of its entries: $\deg(x)=\sum_{v\in V(G)} x(v)$.
We denote the set of divisors on a digraph $G$ by $\Div(G)$, and the set of divisors 
of degree $k$ by $\Div^k(G)$. Note that $\Div(G)$ and $\Div^0(G)$ are Abelian groups with the coordinatewise addition.

The basic operation in the game is a \emph{firing} of a vertex. For a divisor $x$, firing a vertex $v$ means taking the new divisor 
$x'=x+L_G\mathbf{1}_v$, i.e, $v$ loses $d^+(v)$ chips, and each out-neighbor $u$ of $v$ receives $d(v,u)$ chips.
Note that a firing preserves the degree of the divisor.

The firing of a vertex $v$ is \emph{legal} with respect to the divisor $x$, if $x(v)\geq d^+(v)$, i.e, if the vertex $v$ has a nonnegative number of chips after the firing. (Note that other vertices might have a negative number of chips.)
A \emph{legal game} is a sequence of divisors in which each divisor is obtained from the previous one by a legal firing.

The following equivalence relation on $\Div(G)$, called \emph{linear equivalence}, plays an important role in the theory of chip-firing:
$x\sim y$ if there exists an integer vector $z\in \mathbb{Z}^{V(G)}$ such that $y=x+L_Gz$.
One can easily check that this is indeed an equivalence relation. As $\per_G$ is a strictly positive eigenvector of $L_G$ with eigenvalue zero, we can suppose that $z\geq 0$: We have $L_G(z+k\cdot \per_G)=L_Gz$ for any $k\in\mathbb{Z}$, and for a sufficiently large $k$, $z+k\cdot \per_G\geq 0$. Thus $x\sim y$ if and only if $y$ can be reached from $x$ by a sequence of (not necessarily legal) firings.

Note that the divisors linearly equivalent to $\mathbf{0}_G$ form a subgroup of $\Div^0(G)$ which is isomorphic to $\Im(L_G)$, the image of the linear operator on $\mathbb{Z}^{V(G)}$ corresponding to $L_G$.
The factor group of $\Div^0(G)$ by linear equivalence is called the \emph{Picard-group} of the graph:
\begin{equation*}
\Pic^0(G)=\bigslant{\Div^0(G)}{\Im(L_G)}.
\end{equation*}

\subsection{Rotor-routing}
\label{sec::rr_def}

The rotor-routing game is played on a ribbon digraph.
A \emph{ribbon digraph} is a digraph together with a fixed cyclic ordering of the outgoing edges from $v$ for each vertex $v$.
For an edge $e=\overrightarrow{vw}$, denote by $e^+$ the edge following $e$ in the cyclic order at $v$. From this point, we always assume that our digraphs have a ribbon digraph structure.

Let $G$ be a ribbon digraph.
 A \emph{rotor configuration} on $G$ is a function $\varrho$ that assigns to each non-sink  
vertex $v$ an out-edge with tail $v$. We call $\varrho(v)$ the \emph{rotor} at $v$.
 For a rotor configuration $\varrho$, we call the subgraph with edge set $\{\varrho(v): v\in V(G)\}$ the \emph{rotor subgraph}.

A configuration of the rotor-routing game is a pair $(x,\varrho)$, where $x\in\Div(G)$ is divisor, and $\varrho$ is a rotor configuration on $G$. We also call such pairs \emph{divisor-and-rotor configuration}, or just shortly DRC.

Given a configuration $(x,\varrho)$, a \emph{routing} at vertex $v$ results in the  
configuration $(x', \varrho')$, where
$\varrho'$ is the rotor configuration with
$$
\varrho'(u) = \left\{\begin{array}{cl} \varrho(u) & \text{if $u\neq v$,}  \\
         \varrho(u)^+ & \text{if $u=v$},
      \end{array} \right.
$$
and $x'=x-\mathbf{1}_v+\mathbf{1}_{v'}$ where $v'$ is the head of $\varrho'(v)$.

We call the routing at $v$ \emph{legal} (with respect to the configuration $(x,\varrho)$), if $x(v)>0$, i.e.~the routing at $v$ does not create a negative entry at $v$. Note that other vertices might have a negative number of chips. A \emph{legal game} is a sequence of configurations such that each configuration is obtained from the previous one by a legal routing.


An important special case of the rotor-routing game is when the initial configuration has a nonnegative divisor of degree one, i.e, one vertex has one chip, and the other vertices have zero chips. We call such a configuration a \emph{one chip-and-rotor configuration}.
For such a configuration, there is exactly one vertex at which one can perform a legal routing, namely, the vertex of the chip, and the legal routing again leads to a one chip-and-rotor configuration. Thus, in this case, the rotor-routing game is deterministic. We call this special case the \emph{classical rotor-routing process}.
The \emph{orbit} of a one chip-and-rotor configuration is defined as the set of configurations reachable from it by a legal game.

\section{A characterization of recurrect elements}
\label{sec::rec_char}

Recurrent elements play an important role in the rotor-router dynamics.

\begin{definition}
A divisor-and-rotor configuration $(x,\varrho)$ is \emph{recurrent}, if starting from $(x, \varrho)$, there exists a legal rotor-routing game that leads back to $(x, \varrho)$.
\end{definition}

For the classical rotor-routing process, Holroyd et al.~\cite{Holroyd08} gave a characterization for the recurrent configurations.
To state their result, we need a definition.

\begin{definition}[unicycle \cite{Holroyd08}]
A \emph{unicycle} is a one chip-and-rotor configuration 
 where the rotor subgraph
contains a unique directed cycle, and the chip lies on this cycle.
\end{definition}

\begin{thm}[{\cite[Theorem 3.8]{Holroyd08}}] \label{thm::rec_crc_jell}
If $G$ is strongly connected, then the recurrent one chip-and-rotor configurations
are exactly the unicycles.
\end{thm}

In the following theorem, we generalize this result to the general rotor-routing game.
One of the motivations for characterizing recurrent elements in the rotor-routing game is that for chip-firing, no characterization is known for recurrent divisors on general directed graphs. (The broadest case where a characterization is known is the case of Eulerian digraphs \cite{Perrot}.) Meanwhile, for rotor-routing, the following theorem gives a characterization for recurrent elements on general digraphs. We note that the proof of the ``only if'' direction essentially agrees with the proof of the ``only if'' direction from \cite[Theorem 3.8]{Holroyd08}. Let us call a strongly connected component in a digraph a \emph{sink component}, if there is no edge leaving the component.


\begin{thm} \label{thm::rec_jellemzes} 
For a digraph $G$, a divisor-and-rotor configuration $(x,\varrho)$ is recurrent if and 
only if $G$ has a sink component with vertex set $V_0$ such that $x(v)\geq 0$ for each $v\in V_0$, and on each directed cycle in the rotor subgraph restricted to $V_0$, $(\{\varrho(v): v\in V_0\})$,
there is at least one vertex $v$ with $x(v)>0$.
\end{thm}

\begin{proof}
First we show the ``only if'' direction.
Take a DRC $(x,\varrho)$ which is recurrent.
We claim that there exists a sink component such that $x\geq 0$ restricted to the component.
It is enough to show that in any nonempty legal game that 
transforms $(x,\varrho)$ back to itself, in some sink component, each vertex is routed at least once. Indeed, since we require legal routings, at the time a vertex is routed, it has positive number of chips, and it can never again become negative. 

Since the initial and final rotor configurations 
are the same, each vertex is routed either zero times, or its rotor 
makes at least one full turn. In the latter case, it passes at least one chip to 
each of its out-neighbors. Since the initial and final divisors are 
also the same, if a vertex receives a chip, it needs to be routed. Hence each vertex reachable on directed path from a routed 
vertex is also routed. 
For any vertex $v$, the set of vertices reachable on a directed path from $v$ contains the vertex set of a sink component. Hence there is indeed a sink component $V_0$ such that $x\geq 0$ on its vertices.

We claim that in $V_0$, there is at least one chip on each rotor cycle.
Take the nonempty legal game that transforms $(x,\varrho)$ back to itself. We have proved, that each vertex of $V_0$ is 
routed at least once. Suppose that there is a cycle $C$ in the rotor 
subgraph $\{\varrho(v),v\in V_0\}$, such that $x(v)=0$ for each $v\in V(C)$. 
Take the vertex $v\in V(C)$ that was last routed among the vertices of $C$. 
Since the final rotor configuration is $\varrho$, the last time $v$ was routed, the chip moved to the 
head of $\varrho(v)$. Let us call this vertex $w$.
Note that also $w\in V(C)$. Since originally 
$x(w)=0$, the divisor on $w$ is never negative during the process, 
therefore after routing $v$, $w$ has a positive number of chips. 
Since at the end $w$ has zero chips, $w$ needs to be routed after 
the last routing of $v$, which is a contradiction.

Now we show the ``if'' direction. It is enough to prove that if for a strongly connected digraph $G$, $x\geq 0$, and there is at least one chip on each rotor cycle, then $x$ is recurrent. Indeed, this shows that if $x$ satisfies the conditions of the theorem, then it is recurrent restricted to a sink component.
Moreover, playing a rotor-routing game on a sink component does not modify the divisor-and-rotor configuration outside the component, hence in this case $x$ is recurrent on the whole graph.

So now let our graph be strongly connected. It is enough to show that if for a DRC $(x,\varrho)$, $x\geq 0$, 
and there is exactly one chip on each rotor cycle, then $(x,\varrho)$ is recurrent.
Indeed, if a DRC $(x,\varrho)$ with $x\geq 0$ has at 
least one chip on each rotor cycle, then there is a divisor $x'$ 
with $x\geq x'\geq 0$ that has exactly one chip on each rotor-
cycle. A legal game from $(x',\varrho)$ is 
also a legal game from $(x,\varrho)$, and if starting from $(x',\varrho)$ it leads back to $(x',\varrho)$, then starting from $(x,\varrho)$ it leads back to $(x,\varrho)$.

So take a DRC $(x,\varrho)$ with $x\geq 0$ that has exactly one chip on each rotor cycle.
Give a name to each chip: $c_1,\dots , c_k$. Let their initial vertices be $v_1,\dots, v_k$, respectively.
In the rotor subgraph, each rotor cycle is in a different weakly connected component. Let the vertex set of the weakly connected component of $v_i$ be $V_i$. Then $V(G)=V_1\cup\dots\cup V_k$.
Moreover, $\{\varrho(v): v\in V_i-v_i\}$ is an in-arborescence rooted at $v_i$, that spans $V_i$. Let us call this arborescence $A_i$.

Let us do the following procedure:
For each vertex, remember how many times it has been routed (zero at the beginning). We call a vertex $v$ \emph{finished} at some time step, if it has been routed exactly $d^+(v)\cdot \per_G(v)$ times. Our procedure ensures that no vertex is routed more times than this. 
We start with routing the current vertex of $c_1$, until $c_1$ arrives at a finished vertex. We say that at this moment, $c_1$ gets finished. Then we start routing the vertex of $c_2$ until $c_2$ also arrives at a finished vertex, etc. until $c_k$ also arrives at a finished vertex.
Since we always route the vertex of a chip, we only make legal routings during this procedure.
Also, no vertex $v$ gets routed more than $d^+(v)\cdot \per_G(v)$ times, since whenever a chip arrives at a finished vertex, we stop routing it.

It is enough to show that during this procedure, each vertex $v$ is routed 
exactly $d^+(v)\cdot \per_G(v)$ times. From this, it follows immediately that at the end of the process, we arrive back to $(x,\varrho)$, as then each rotor makes some full turns, and each vertex $v$ forwards $d^+(v)\cdot\per_G(v)$ chips, and receives $\sum_{u\in \Gamma^-(v)}\per_G(u)$ chips. The two quantities are equal because $L_G\per_G=\mathbf{0}_G$.

We show by induction, that by the time $c_1,\dots, c_i$ are finished, all vertices in $V_1 \cup \dots \cup V_i$ are finished, and $c_j$ is in $v_j$ for $j=1,\dots, i$. For $i=k$, this proves that $(x,\varrho)$ is indeed recurrent.

For $i=0$, the condition is meaningless. Suppose that the condition holds for some $i-1$. We show that it also holds for $i$.

Since by induction hypothesis, $c_1,\dots, c_{i-1}$ all got finished in their initial positions, before we start routing $c_i$, all vertices forwarded and received the same number of chips.
Thus, while we are routing $c_i$, if at some moment $c_i$ is at a vertex $v\neq v_i$, then each vertex $u\notin \{v,v_i\}$ received and forwarded the same number of chips, $v$ received one more chips than forwarded, and $v_i$ forwarded one more chips than received. If $c_i$ is at $v_i$, then each vertex received and forwarded the same number of chips.

Suppose that the first finished vertex reached by $c_i$ is $v$. Then $v$ has been routed $d^+(v)\cdot \per_G(v)$ times. Since any in-neighbor $u$ of $v$ has been routed at most $d^+(u)\cdot \per_G(u)$ times, any such in-neighbor forwarded at most $\per_G(u)$ chips to $v$. Hence $v$ received at most 
$\sum_{u\in \Gamma^-(v)}\per_G(u)=d^+(v) \per_G(v)$ chips.
Thus when $c_i$ first reached $v$ as a finished vertex, $v$ received at most as many chips, as it forwarded. Hence $v = v_i$.

We show that each vertex in $V_i$ gets finished by the time $c_i$ gets finished. Since $A_i$ is an in-arborescence rooted at $v_i$ spanning $V_i$, it is enough to show, that when a vertex $v$ receives the chip for the $d^+(v) \per_G(v)$-th time, each of its in-neighbors in $A_i$ are already finished.

Suppose that $v$ has just received a chip for the $d^+(v) \per_G(v)$-th time. As it received at most $\sum_{u\in\Gamma^-(v)}\per_G(u)=\per_G(v) d^+(v)$ chips from its in-neighbors, to have equality, $v$ must have received $\per_G(u)$ chips from each in-neighbor $u$. But for those in-neighbors $u$, where $\overrightarrow{uv}\in A_i$, the chip is forwarded towards $v$ for the $d^+(u)$-th, $2d^+(u)$-th, $\dots$ times, so from these vertices, a chip must have been forwarded $\per_G(u) d^+(u)$ times, hence they are indeed finished.
\end{proof}

\begin{coroll} \label{cor::rec_deg_one=unicycle}
On strongly connected digraphs, the recurrent configurations where the degree of the divisor is one are exactly the unicycles.
\end{coroll}

From the proof of Theorem \ref{thm::rec_jellemzes}, we can easily deduce a formula for the possible lengths of legal games transforming a recurrent DRC back to itself:

\begin{prop} \label{prop::altalanos_perhossz}
For a strongly connected digraph $G$, if a DRC $(x,\varrho)$ is recurrent, then for any nonempty legal game that transforms it back to itself, there is an integer $k\in \mathbb{N}$ such that each vertex $v$ is routed $k\cdot d^+(v)\cdot\per_G(v)$ times. Moreover, there exists a legal game with $k=1$. 
\end{prop}
\begin{proof}
If a legal game transforms a DRC $(x,\varrho)$ back to itself, then each rotor makes some full turns. Thus for each $v\in V(G)$, there exists some $z(v)\in \mathbb{N}$ such that $v$ has been routed $d^+(v)\cdot z(v)$ times.
Since the initial and final divisors are also the same, each vertex gave and received the same number of chips. If a vertex $u$ was routed $z(u)\cdot d^+(u)$ times, a vertex $v\in \Gamma^+(u)$ received $z(u)$ chips from it. Thus for each vertex $v$,
$$
\sum_{u\in\Gamma^-(v)}z(u)=z(v)\cdot d^+(v).
$$
Hence the vector $z$ is an eigenvector of the Laplacian matrix with eigenvalue zero. Since $L_G$ has a one-dimensional kernel,  $z$ is a multiple of $\per_G$.

The construction in the proof of Theorem \ref{thm::rec_jellemzes} shows that for any recurrent DRC, there exists a legal game with $k=1$ that transforms it back to itself.
\end{proof}

For a unicycle, the rotor-routing game is deterministic, hence we obtain that it takes $\sum_{v\in V(G)}\per_G(v) d^+(v)$ steps for the rotor-router process to return to the initial configuration.
This gives the following theorem, originally proved by Pham \cite{pham_rotor} using linear algebra.

\begin{thm} \label{thm::perhossz}
    For a strongly connected digraph $G$, the size of the orbit of any unicycle is $\sum_{v\in V(G)}\per_G(v) d^+(v)$.
\end{thm}

\section{Linear equivalence}
\label{sec::lin_eqv}

For the chip-firing game, linear equivalence is a linear-algebraic type, computationally well-behaved concept, that proves very useful for analyzing reachability questions. In this section, we generalize the concept of linear equivalence to the rotor-routing game. Then we apply it to analyzing reachability questions in the rotor-routing game, and to give a new interpretation of the rotor-routing action of the Picard group on the set of spanning in-arborescences.
Using the connection between linear equivalence for chip-firing and for rotor-routing, we give a bijective proof that the number of rotor-router unicycle-orbits equals the order of the Picard group.

In chip-firing, for strongly connected digraphs, linear equivalence is the same as reachability where we let non-legal firings to happen. For rotor-routing, we use the analogue of this characterization as definition.
Let us call a non-necessarily legal routing an \emph{unconstrained routing}.

\begin{definition}[linear equivalence of divisor-and-rotor configurations]
We define two configurations $(x_1,\varrho_1)$ and $(x_2,\varrho_2)$ to be linearly equivalent, if $(x_2,\varrho_2)$ can be reached from $(x_1,\varrho_1)$ by a sequence of unconstrained routings. We denote this by $(x_1,\varrho_1)\sim (x_2,\varrho_2)$.
\end{definition}

\begin{remark} \label{rem::elozmeny_Levine}
The idea of analyzing the interplay between legal and non-legal games has appeared previously in some papers. See for example \cite{height-arrow04,Kager-Levine,Holroyd08}.
\end{remark}

\begin{remark} \label{rem::routing_tulaljdonsagok}
Suppose we have an initial configuration, and a multiset of vertices to perform unconstrained routings at. Then the resulting configuration is independent of the order in which we perform the routings.
Hence we can encode a sequence of unconstrained routings in a vector $r\in \mathbb{N}^{V(G)}$ such that $r(v)$ is the number of times vertex $v$ has been routed. We call such a vector a \emph{routing vector}.

Similarly, for chip-firing, by firing a vector $z\in\mathbb{N}^{V(G)}$, we mean firing each vertex $v$ $z(v)$ times. This has the effect of adding $L_Gz$ to the divisor, independent of the order in which we perform the (not necessarily legal) firings.

Note that if a routing vector $r$ is of the form $r=(d^+(v_1)\cdot z(v_1), \dots, d^+(v_n)\cdot z(v_n))$ for some $z\in \mathbb{N}^{V(G)}$, then routing $r$ from a DRC $(x,\varrho)$ leads to a DRC $(x',\varrho)$, where $x'$ is the divisor we get after firing the vector $z$ from $x$.
\end{remark}

\begin{prop}
On strongly connected digraphs, linear equivalence of divisor-and-rotor configurations is an equivalence-relation.
\end{prop}
\begin{proof}
Reflexivity and transitivity are obvious. Let us prove symmetry.
It is enough to prove that if we get $(x_2,\varrho_2)$ from $(x_1,\varrho_1)$ by one unconstrained routing at a vertex $w$, then $(x_1,\varrho_1)$ can also be reached from $(x_2,\varrho_2)$ by unconstrained routings. Take the following routing vector $r$:
$$
r(v) = \left\{\begin{array}{cl} d^+(v)\per_G(v) & \text{if $v\neq w$},  \\
         d^+(w)\per_G(w)-1 & \text{if $v=w$}.
      \end{array} \right.
$$
This is nonnegative, as $\per_G$ has strictly positive coordinates for strongly connected digraphs.
Routing $r$ from $(x_2,\varrho_2)$ is equivalent to routing $(d^+(v_1)\cdot \per_G(v_1), \dots, d^+(v_n)\cdot \per_G(v_n))$ from $(x_1,\varrho_1)$, that by Remark \ref{rem::routing_tulaljdonsagok} leads to $(x_1,\varrho_1)$.
\end{proof}

The following lemma shows the connection between the linear equivalence of divisor-and-rotor configurations, and the linear equivalence of graph divisors.

\begin{lemma} \label{lem::div_ekv_jell}
Let $G$ be a strongly connected digraph. If $\varrho$ is a rotor configuration, and $x_1, x_2 \in \Div(G)$, then $x_1\sim x_2$ if and only if $(x_1,\varrho)\sim(x_2,\varrho)$.
\end{lemma}
\begin{proof}
 Suppose that $x_1\sim x_2$. This means that there exists $z\in \mathbb{Z}^{V(G)}$ such that $x_2=x_1 + L_Gz$. Moreover, we can suppose that $z$ is nonnegative, otherwise we can add $\per_G$ to it sufficiently many times.
 From initial configuration $(x_1,\varrho_1)$, route vertices according to the following routing vector: $(d^+(v_1)\cdot z(v_1), \dots, d^+(v_n)\cdot z(v_n))$. Then the resulting chip-moves are exactly the same as in chip-firing after firing the vector $z$, thus we arrive at the divisor $x_2$. On the other hand, each rotor made some full turns, hence the final rotor configuration is again $\varrho$.
 
 Now suppose that $(x_1,\varrho)\sim (x_2,\varrho)$. Fix a routing vector $r$ witnessing the equivalence of $(x_1,\varrho)$ and $(x_2,\varrho)$. Then since the initial and the final rotor configurations are both $\varrho$, each rotor made some full turns, hence $r$ must be of the form $r=(d^+(v_1)\cdot z(v_1), \dots, d^+(v_n)\cdot z(v_n))$ for some $z\in \mathbb{Z}^{V(G)}$. Then firing $z$ induces the same chip-moves as routing $r$, hence $x_2=x_1+L_Gz$, thus $x_1\sim x_2$.
\end{proof}

Maybe the nicest property of the linear equivalence is that it is computationally well-behaved:

\begin{prop} \label{prop::lin_eqv_decidable}
For given divisor-and-rotor configurations $(x_1,\varrho_1)$ and $(x_2,\varrho_2)$, deciding whether $(x_1,\varrho_1)\sim(x_2,\varrho_2)$ holds can be done in polynomial time.
\end{prop}
\begin{proof}
For each vertex $v$, let $\alpha(v)$ be the number of out-edges from $v$ such that $\varrho_1(v)<e\leq \varrho_2(v)$ in the cyclic order at $v$. If we route the routing vector $\alpha$ from $(x_1,\varrho_1)$, we arrive at a configuration $(y,\varrho_2)$, where $y$ is some divisor. This means at most $|E(G)|$ routings.
For digraphs with multiple edges, $|E(G)|$ is not necessarily polynomial in the size of the input, but note that for each pair of vertices $u,v\in V(G)$, we can compute how many chips need to pass through the multi-edge $\overrightarrow{uv}$, and we can do this in time linear in the size of the description of the cyclic order at $u$. Hence we can compute the divisor $y$ in polynomial time.

As $(y,\varrho_2)\sim (x_1,\varrho_1)$, we have $(x_1,\varrho_1)\sim(x_2,\varrho_2)$ if and only if $(y,\varrho_2)\sim(x_2,\varrho_2)$, which by Lemma \ref{lem::div_ekv_jell} is equivalent to $y\sim x_2$. This can be checked in polynomial time using Gaussian elimination, then solving a system of linear congruence equations (see also \cite[Proposition 8]{chip-reach}).
\end{proof}

\subsection{Reachability questions}

\begin{notation}
Let us denote by $(x_1, \varrho_1)\leadsto(x_2, \varrho_2)$ if $(x_2, \varrho_2)$ can be reached from $(x_1, \varrho_1)$ by a legal rotor-routing game.
\end{notation}

In this section, we examine the reachability problem for rotor-routing from a computational aspect. As Theorem \ref{thm::perhossz} shows, in the classical rotor-routing process, unicycle-orbits can have exponential size. Hence there exist configurations such that one is only reachable from the other by exponentially many routings. This shows that the question of deciding whether one divisor-and-rotor configuration can be reached from another one by a legal rotor-routing game is nontrivial.
However, as the following proposition shows, if the target configuration is recurrent, the reachability problem is decidable in polynomial time. This result is an analogue of a recent result for the chip-firing game \cite{chip-reach}. The proof is also a complete analogue.

\begin{prop} \label{prop::diffuz_DRC_elerheto}
Let $(x_1,\varrho_1)$ and $(x_2,\varrho_2)$ be two divisor-and-rotor configurations on a strongly connected digraph.
If $(x_2,\varrho_2)$ is recurrent, then
$(x_1, \varrho_1)\leadsto(x_2, \varrho_2)$
if and only if $(x_1,\varrho_1)\sim (x_2,\varrho_2)$ .
\end{prop}
\begin{proof}
  The ``only if'' direction is obvious, since a sequence of legal routings is also a sequence of unconstrained routings.

  Let us prove the ``if'' direction. By our assumption, $(x_2, \varrho_2)$ is recurrent. Let $v_1, v_2, \dots , v_m$ be a sequence of vertices such that routing them in this order is a legal rotor-routing game that transforms $(x_2, \varrho_2)$ back to itself. By Proposition \ref{prop::altalanos_perhossz}, we can suppose that in this sequence each vertex $v$ occurs $d^+(v)\per_G(v)$ times. As $\per_G$ is strictly positive, this means that each vertex occurs at least once.

  By our assumption that $(x_1,\varrho_1)\sim (x_2,\varrho_2)$, there exists a routing vector $r\in \mathbb{N}^{V(G)}$ such that routing $r$ transforms $(x_2,\varrho_2)$ to $(x_1,\varrho_1)$.
  
  We proceed by induction on $|r|=\sum_{v\in V}r(v)$. If $|r|=0$, then $(x_1,\varrho_1)= (x_2,\varrho_2)$ hence we have nothing to prove.
  Otherwise let $i$ be the smallest index such that $r(v_i)>0$. Such an index exists since each vertex occurs in the sequence $v_1,\dots ,v_m$.
  From $(x_2,\varrho_2)$ route at vertices $v_1, \dots, v_{i-1}$.  These are all legal routings by definition. Let the resulting DRC be $(x'_2, \varrho'_2)$. 
  
  We claim that routing $v_1, \dots, v_{i-1}$ from $(x_1,\varrho_1)$ is also a legal game. Indeed, as $r(v_1)=\dots=r(v_{i-1})=0$, we can get $(x_1,\varrho_1)$ from $(x_2,\varrho_2)$ such that we do not route at $v_1, \dots, v_{i-1}$. Hence $x_1(v_j)\geq x_2(v_j)$ for $j=1,\dots, i-1$. Also for the same reason, $\varrho_1(v_j)=\varrho_2(v_j)$ for $j=1,\dots, i-1$. Hence for the two initial configurations, while routing $v_1,\dots v_j$ for some $1\leq j\leq i-1$, the chip-moves, and the rotor-moves are the same in both games. Therefore at any time step, the rotors at $v_1,\dots v_{i-1}$ are the same in the two games, and the number of chips is greater or equal on these vertices in the game with initial configuration $(x_1,\varrho_1)$. This shows that routing $v_1, \dots, v_{i-1}$ from $(x_1,\varrho_1)$ is indeed a legal game.
Let the resulting DRC be $(x'_1, \varrho'_1)$. Then $(x_1,\varrho_1)\leadsto (x'_1, \varrho'_1)$.
  
  From initial configuration $(x_2,\varrho_2)$, routing $r$ then routing $v_1, \dots, v_{i-1}$ has the same effect as routing $v_1, \dots, v_{i-1}$ then routing $r$. Routing $r$ transforms $(x_2,\varrho_2)$ to $(x_1,\varrho_1)$, then routing $v_1, \dots, v_{i-1}$ transforms that to $(x'_1,\varrho'_1)$. Routing $v_1, \dots, v_{i-1}$ from $(x_2,\varrho_2)$ transforms it to $(x'_2, \varrho'_2)$. We conclude that routing $r$ from $(x'_2, \varrho'_2)$ results in $(x'_1, \varrho'_1)$. 
  
  From $(x'_2, \varrho'_2)$, route $v_i$ (this is also a legal routing). Let the resulting DRC be $(x''_2, \varrho''_2)$. Then for 
  $$
  r'(v)= \left\{\begin{array}{cl} r(v) & \text{if }v\neq v_i,  \\
         r(v_i)-1 & \text{if } v=v_i,
      \end{array} \right.
$$
we have that routing $r'$ from $(x''_2, \varrho''_2)$ results in $(x'_1, \varrho'_1)$, moreover, $|r'|=|r|-1$.

We claim that $(x''_2, \varrho''_2)$ is also a recurrent DRC. Indeed, routing vertices $v_{i+1}, \dots , v_m, v_1, \dots , v_i$ is a legal game that transforms $(x''_2, \varrho''_2)$ to itself.
 Hence by induction hypothesis, $(x'_1, \varrho'_1)\leadsto(x''_2, \varrho''_2)$. As $(x_1, \varrho_1)\leadsto (x'_1, \varrho'_1)$ 
and $(x''_2, \varrho''_2)\leadsto(x_2, \varrho_2)$, we have $(x_1, \varrho_1)\leadsto(x_2, \varrho_2)$.
\end{proof}

\begin{coroll} \label{cor::unicyc_orbit_jell}
Two unicycles $(\mathbf{1}_{v_1},\varrho_1)$ and $(\mathbf{1}_{v_2},\varrho_2)$ lie in the same rotor-router orbit if and only if $(\mathbf{1}_{v_1},\varrho_1)\sim(\mathbf{1}_{v_2},\varrho_2)$.
\end{coroll}

Corollary \ref{cor::unicyc_orbit_jell} together with Proposition \ref{prop::lin_eqv_decidable} gives us the following:

\begin{prop}
It can be decided in polynomial time whether two unicycles lie in the same rotor-router orbit.
\end{prop}

\subsection{The number of unicycle-orbits}

In this section, we give a simple bijective proof for the fact that the number of unicycle-orbits of the classical rotor-routing process equals to the order of the Picard group. We note that this statement has been known, it follows from the combination of \cite[Theorem 1]{pham_rotor} and \cite[Theorem 2.10]{farrell-levine-coeulerian} and the author also gave a less direct proof for it in an earlier manuscript \cite[Proposition 2.6]{rotor_earlier}. However these previous proofs did not provide a bijection. Let us first state a technical lemma.
\begin{lemma} \label{lem::van_rekurrens}
 Each equivalence class of divisor-and-rotor configurations where the degree of the divisor is at least one contains a recurrent configuration.
\end{lemma}
\begin{proof}
Take a DRC $(x, \varrho)$ with $\deg(x)=k\geq 1$. As $\deg(x)\geq 1$, there exists a vertex $v$ with $x(v)>0$. Make a routing at $v$. As the resulting divisor still has degree $k$, there is once again a vertex with positive number of chips. For this reason, we can play a legal game as long as we wish. 
Let $l=\sum_{v\in V(G)} x(v)^+$, where $x(v)^+=max\{0,x(v)\}$. 
If at the beginning, a vertex had $x(v)=t<0$, then after some legal routings, its number of chips is necessarily larger or equal to $t$. (While it is negative, it is never routed, and if it ever becomes nonnegative, it never again becomes negative.)
For the same reason, at any time, on any vertex, the number of chips is at most $l$. 

This means that there are only finitely many configurations we can reach from $(x, \varrho)$; therefore, after finitely many steps, we get some configuration for the second time. This one will be recurrent.
\end{proof}

\begin{coroll} \label{cor::DRC_osztalyok_es_rr_orbitok_szama}
 For a strongly connected digraph, the number of DRC-equivalence classes of degree one equals the number of rotor-router unicycle-orbits.
\end{coroll}

\begin{prop} \label{prop::orbitszam_es_Picard}
For a strongly connected digraph $G$, the order of $\Pic^0(G)$ equals the number of rotor-router unicycle-orbits.
\end{prop}

\begin{proof}
 By Corollary \ref{cor::DRC_osztalyok_es_rr_orbitok_szama}, the number of rotor-router unicycle-orbits equals the number of DRC equivalence classes of degree one.

 The order of the Picard group is by definition the number of equivalence classes of degree zero divisors. The number of equivalence classes of degree zero divisors equals the number of equivalence classes of degree one divisors, as for an arbitrary fixed vertex $v$, $x\mapsto x+1_v$ is a bijection between degree zero and degree one divisors that is compatible with the linear equivalence.
 
 Hence we need to show that the number of DRC-equivalence classes of degree one equals the number of divisor equivalence classes of degree one. Let us denote the equivalence class of a divisor $x$ by $[x]$, and the equivalence class of a divisor-and-rotor configuration $(x,\varrho)$ by $[(x,\varrho)]$. 
Fix a rotor configuration $\varrho$. 
We show that the map $[x]\to [(x,\varrho)]$, where $x$ is a divisors of degree one, is well defined, and gives a bijection between divisor classes of degree one and DRC classes of degree one.

Lemma \ref{lem::div_ekv_jell} ensures that the mapping is well defined and injective. Moreover, each DRC equivalence class contains at least one DRC with rotor configuration $\varrho$, since from an arbitrary DRC we can route each vertex the required number of times. Hence the mapping is also surjective.
\end{proof}

\subsection{The rotor-router action}

In this section, let $G$ be an Eulerian digraph.
Holroyd et al.~\cite{Holroyd08} defined a group action of the Picard group on the spanning in-arborescences of the graph, using the rotor-router operation. 
We give an interpretation of this group action in terms of the equivalence classes of divisor-and-rotor configurations.

\begin{notation}
We denote the set of spanning in-arborescences of $G$ rooted at $r$ by $\Arb(G,r)$.
For a $T\in\Arb(G,r)$, let us denote by $T(v)$ the edge leaving node $v\neq r$.

For any fixed edge $\overrightarrow{rw}$, the following mapping $\varrho$ is a rotor configuration with exactly one cycle:
$$
\varrho(v) = \left\{\begin{array}{cl} T(v) & \text{if $v\neq r$,}  \\
         \overrightarrow{rw} & \text{if $v=r$}.
      \end{array} \right.
$$
Let us denote $\varrho=T\cup\overrightarrow{rw}$.
\end{notation}

\begin{definition}[Rotor-router action, \cite{Holroyd08}]\label{def::rr_action}
The \emph{rotor-router action} is defined with respect to a base vertex $r\in V(G)$ that we call the \emph{root}. It is a group action of $\Pic^0(G)$ on the spanning in-arborescences of $G$ rooted at $r$.
We denote by $x_r(T)$ the image of a $T\in\Arb(G,r)$ at the action of the divisor $x\in\Div^0(G)$.

$x_r(T)$ is defined as follows: Choose a divisor $x'\sim x$ such that $x'(v)\geq 0$ for each $v\neq r$. Such an $x'$ can easily be seen to exist.
Fix any out-edge $\overrightarrow{rw}$ of $r$.
Let $\varrho=T\cup\overrightarrow{rw}$. 
Start a legal rotor-routing game from $(x',\varrho)$,
such that $r$ is not allowed to be routed.
Continue until each chip arrives at $r$. Holroyd et al.~\cite{Holroyd08} shows that this procedure ends after finitely many steps, and in the final configuration $(\mathbf{0}_G,\varrho')$, the edges $\{\varrho'(v):v\in V(G)-r\}$ form a spanning in-arborescence of $G$ rooted at $r$. $x_r(T)$ is defined to be this arborescence.
\end{definition}

Holroyd et al.~\cite{Holroyd08} shows that for Eulerian digraphs, $x_r(T)$ is well defined, i.e.,~the definition does not depend on our choice of $x'$ and on the choice of the legal game. From this, it also follows that this is indeed a group action of $\Pic^0(G)$ on $\Arb(G,r)$, i.e.,~$x_r(T)=x'_r(T)$ if $x\sim x'$. Note that the choice of $w$ is immaterial in the construction.

Now we give an alternative definition of this group action using the notion of linear equivalence. First we need a technical lemma.

\begin{notation}
Let us call a divisor-and-rotor configuration \emph{$\overrightarrow{rw}$-good}, if it is of the form $(\mathbf{0_G}, \varrho)$, where $\varrho(r)=\overrightarrow{rw}$, and the edges $\{\varrho(v): v\in V(G)-r\}$ form a spanning in-arborescence of $G$ rooted at $r$. 
\end{notation}

\begin{lemma}
For a strongly connected Eulerian digraph $G$, vertex $r\in V(G)$ and edge $\overrightarrow{rw}\in E(G)$, in each DRC equivalence class of degree zero, there is exactly one $\overrightarrow{rw}$-good DRC.
\end{lemma}

\begin{proof}
Take a DRC equivalence class $C$ of degree zero. Let 
$$
C+\mathbf{1}_r=\{(x+\mathbf{1}_r,\varrho): (x,\varrho)\in C\}.
$$
This is a DRC equivalence class of degree one.
A configuration $(\mathbf{0}_G,\varrho)\in C$ is $\overrightarrow{rw}$-good if and only if $(\mathbf{1}_r,\varrho)\in C+\mathbf{1}_r$ is a unicycle with $\varrho(r)=\overrightarrow{rw}$. Thus, it is enough to show, that in each DRC equivalence class of degree one, there is exactly one unicycle $(\mathbf{1}_r,\varrho)$ where $\varrho(r)=\overrightarrow{rw}$.

By Lemma \ref{lem::van_rekurrens}, there exists a recurrent element in each DRC equivalence class of degree one, which is a unicycle by Corollary \ref{cor::rec_deg_one=unicycle}.
If we run the rotor-router process from this unicycle until it returns to the initial position, each vertex $v$ is visited $d^+(v)$ times by the chip. Therefore, $r$ is reached by the chip $d^+(r)$ times, and during these visits, the rotor at $r$ turns around. Hence there will be a moment, when the chip is at $r$, and the rotor at $r$ is $\overrightarrow{rw}$. As the rotor-router process takes unicycles to unicycles \cite[Lemma 3.3]{Holroyd08}, this is going to be a unicycle of the form $(\mathbf{1}_r,\varrho)$ where $\varrho(r)=\overrightarrow{rw}$.

Now suppose there are two linearly equivalent unicycles $(\mathbf{1}_r,\varrho_1)$ and $(\mathbf{1}_r,\varrho_2)$ with $\varrho_1(r)=\varrho_2(r)=\overrightarrow{rw}$. Then by 
Corollary \ref{cor::unicyc_orbit_jell}, they lie in the same rotor-router orbit. We get all elements of the orbit of $(\mathbf{1}_r,\varrho_1)$ by running the rotor-router process started from $(\mathbf{1}_r,\varrho_1)$ until it arrives back to $(\mathbf{1}_r,\varrho_1)$. During this process, the chip visits the vertex $r$ only $d^+(r)$ times, hence there is only one unicycle in the orbit that has its chip in $r$ with $\overrightarrow{rw}$ as rotor. As $(\mathbf{1}_r,\varrho_1)$ and $(\mathbf{1}_r,\varrho_2)$ are both in the orbit, this means that they must be the same unicycle.
\end{proof}

\begin{definition}[Alternative definition of the rotor-router action] \label{d:alt_def_rotor_action}
Let a divisor $x$ of degree zero act on a spanning in-arborescence $T$ rooted at $r\in V(G)$ as follows:

Fix a vertex $w$ such that $\overrightarrow{rw}\in E(G)$.
Let $\varrho=T \cup \overrightarrow{rw}$.

Let $(\mathbf{0}_G,\varrho')$ be the unique $\overrightarrow{rw}$-good DRC equivalent to $(x,\varrho)$. Let $T'$ be the spanning in-arborescence $\{\varrho'(v): v\in V(G)-r\}$.
Then let $T^x=T'$.
\end{definition}

\begin{prop}
$x_r(T)=T^x$ for any choice of $r\in V(G)$, $x\in \Div^0(G)$ and $T\in \Arb(G,r)$.
\end{prop}

\begin{proof}
Let $\varrho=T \cup \overrightarrow{rw}$.
In the construction of Definition \ref{def::rr_action}, we obtain a DRC $(\mathbf{0_G}, \varrho')$ linearly equivalent to $(x,\varrho)$, where $\{\varrho'(v):v\in V(G)-r\}$ is a spanning in-arborescence. Moreover, since $r$ is not routed during the process, $\varrho'(r)=\overrightarrow{rw}$. Hence $(\mathbf{0_G}, \varrho')$ is $\overrightarrow{rw}$-good, so both definitions give the spanning in-arborescence $\{\varrho'(v): v\in V(G)-r\}$.
\end{proof}

\begin{thm}
For an Eulerian digraph G, given two spanning in-arbo\-res\-cen\-ces $T_1,T_2\in \Arb(G,r)$ and a divisor $x\in\Div^0(G)$, it can be decided in polynomial time whether $x_r(T_1)=T_2$.
\end{thm}
\begin{proof}
One needs to check whether for an out-edge $\overrightarrow{rw}$ of $r$, $(x,T_1\cup\overrightarrow{rw})\sim(\mathbf{0}_G,T_2\cup\overrightarrow{rw})$. This can be done in polynomial time by Proposition \ref{prop::lin_eqv_decidable}.
\end{proof}

\subsubsection{Base-point independence of the rotor-router action}
Let us turn to undirected graphs (that we simultaneously imagine as directed graphs where each undirected edge is replaced by two oppositely directed edges). For undirected graphs, the spanning in-arborescences with root $r$ are in one-to-one correspondence with the spanning trees. Therefore, we can think of the rotor-router action with base point $r$ as an action on the spanning trees of the graph. Since now the rotor-router action with any base vertex acts on the same set of objects, one can ask for which ribbon graphs is the action independent of the base vertex.
Chan, Church and Grochow \cite{Chan15} shows that the rotor-router action is independent of the base vertex if and only if the ribbon graph is planar. (Planarity for a ribbon graph means that the ribbon graph structure gives a combinatorial embedding of the graph into the plane.)
Their proof proceeds in two steps. First they show the following:

\begin{notation}
For a rotor configuration $\varrho$, let $\overleftarrow{\varrho}$ be the rotor configuration in which each rotor cycle is reversed, and all other rotors are left the same. Let us call this the reversal of the rotor configuration. See Figure 1 for an example.
\end{notation}

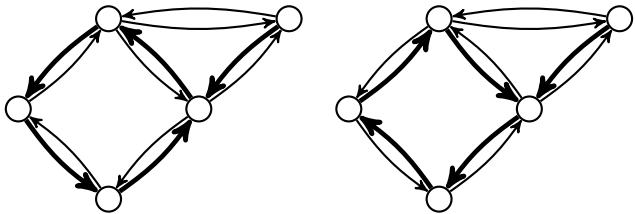
\begin{figure}[ht] \label{fig::reversal}
\begin{center}
\begin{tikzpicture}[->,>=stealth',auto,scale=1.2,
                    thick,every node/.style={circle,draw,font=\sffamily\small}]
  \node[] (1) at (1, 0) {};  
  \node[] (2) at (0, 1) {};
  \node[] (3) at (-1, 0) {};  
  \node[] (4) at (0, -1) {};
  \node[] (5) at (2, 1) {};
  \path[every node/.style={font=\sffamily\small},line width=1.8pt]
    (1) edge [bend right=10] node {} (2)
    (2) edge [bend right=10] node {} (3)
    (3) edge [bend right=10] node {} (4)
    (4) edge [bend right=10] node {} (1)
    (5) edge [bend right=10] node {} (1);
  \path[every node/.style={font=\sffamily\small},line width=0.8pt]
    (2) edge [bend right=10] node {} (1)
    (3) edge [bend right=10] node {} (2)
    (4) edge [bend right=10] node {} (3)    
    (1) edge [bend right=10] node {} (4)    
    (1) edge [bend right=10] node {} (5)
    (5) edge [bend right=10] node {} (2)    
    (2) edge [bend right=10] node {} (5);  
\end{tikzpicture}
\hspace{0.2cm}
\begin{tikzpicture}[->,>=stealth',auto,scale=1.2,
                    thick,every node/.style={circle,draw,font=\sffamily\small}]
  \node[] (1) at (1, 0) {};  
  \node[] (2) at (0, 1) {};
  \node[] (3) at (-1, 0) {};  
  \node[] (4) at (0, -1) {};
  \node[] (5) at (2, 1) {};
  \path[every node/.style={font=\sffamily\small},line width=1.8pt]
    (2) edge [bend right=10] node {} (1)
    (3) edge [bend right=10] node {} (2)
    (4) edge [bend right=10] node {} (3)    
    (1) edge [bend right=10] node {} (4)
    (5) edge [bend right=10] node {} (1);
  \path[every node/.style={font=\sffamily\small},line width=0.8pt]
    (1) edge [bend right=10] node {} (2)
    (2) edge [bend right=10] node {} (3)
    (3) edge [bend right=10] node {} (4)
    (4) edge [bend right=10] node {} (1)
    (5) edge [bend right=10] node {} (1)
    (1) edge [bend right=10] node {} (5)
    (5) edge [bend right=10] node {} (2)    
    (2) edge [bend right=10] node {} (5); 
\end{tikzpicture}
\end{center}
\caption{A rotor configuration and its reversal.
The rotor edges are drawn by thick lines. }
\end{figure}

\begin{prop}\cite{Chan15}
A connected ribbon graph $G$ without loops is planar if and only if for any unicycle $(\mathbf{1}_v,\varrho)$, $(\mathbf{1}_v,\varrho)\leadsto(\mathbf{1}_v,\overleftarrow{\varrho})$.
\end{prop}

The second step in their proof is to show that the rotor-router action is independent of the base vertex if and only if for any unicycle $(\mathbf{1}_v,\varrho)$, $(\mathbf{1}_v,\varrho)\leadsto(\mathbf{1}_v,\overleftarrow{\varrho})$. We give a simple proof for this second statement using the interpretation of the rotor-routing action in terms of the equivalence classes of divisor-and-rotor configurations.

\begin{prop} \cite{Chan15}
The rotor-router action is independent of the base vertex if and only if for any unicycle $(\mathbf{1}_v,\varrho)$, $(\mathbf{1}_v,\varrho)\leadsto(\mathbf{1}_v,\overleftarrow{\varrho})$.
\end{prop}

\begin{proof}
First we show the ``if'' part.
From Proposition \ref{prop::diffuz_DRC_elerheto}, the fact that for any unicycle $(\mathbf{1}_v,\varrho)$, $(\mathbf{1}_v,\varrho)\leadsto(\mathbf{1}_v,\overleftarrow{\varrho})$ is equivalent to the fact that for any unicycle $(\mathbf{1}_v,\varrho)$, $(\mathbf{1}_v,\varrho)\sim(\mathbf{1}_v,\overleftarrow{\varrho})$, which is in turn equivalent to the fact that 
\begin{equation*} \label{eq::megfordithatosag}
\textrm{for any DRC $(0_G,\varrho)$, where $\varrho$ has exactly one cycle, $(\mathbf{0}_G,\varrho)\sim(\mathbf{0}_G,\overleftarrow{\varrho})$.}
\end{equation*}
Note also, that this condition implies $(x,\varrho)\sim(x,\overleftarrow{\varrho})$ for any $x\in \Div(G)$. 

Since our graph is connected, it is enough to show, that for any two adjacent vertices $v,w\in V(G)$, the rotor-router action with base vertex $v$ equals the rotor-router action with base vertex $w$.

For a vertex $u$ and spanning tree $T$, let us denote by $T_u$ the spanning in-arborescence rooted at $u$ that we get by orienting each edge of $T$ towards $u$.

Take any spanning tree $T$ of $G$, and a divisor $x\in \Div^0(G)$. 
$T_v\cup \overrightarrow{vw}$ is a rotor-configuration, since $v$ and $w$ are adjacent.
By Definition \ref{d:alt_def_rotor_action}, $x_v(T)=T'$ where $(\mathbf{0}_G, T'_v\cup \overrightarrow{vw})\sim (x, T_v\cup \overrightarrow{vw})$.
As $\overleftarrow{T_v\cup \overrightarrow{vw}}=T_w\cup \overrightarrow{wv}$ and $\overleftarrow{ T'_v\cup \overrightarrow{vw}}=T'_w\cup \overrightarrow{wv}$, we have
$(\mathbf{0}_G, T'_v\cup \overrightarrow{vw})\sim (\mathbf{0}_G, T'_w\cup \overrightarrow{wv})$ and 
$(x, T_v\cup \overrightarrow{vw})\sim (x, T_w\cup \overrightarrow{wv})$. Hence by transitivity, 
$(\mathbf{0}_G, T'_w\cup \overrightarrow{wv})\sim (x, T_w\cup \overrightarrow{wv})$. Thus $x_w(T)=T'$.

Now we show that if we have a DRC $(\mathbf{0}_G,\varrho)$, where $\varrho$ has exactly one cycle, such that $(\mathbf{0}_G,\varrho)\not\sim(\mathbf{0}_G,\overleftarrow{\varrho})$, then there exists $v,w\in V(G)$, $x\in\Div^0(G)$ and a spanning tree $T$, such that $x_v(T)\neq x_w(T)$.

Let $v$ be a vertex on the cycle of $\varrho$, and let $w$ be the vertex such that $\varrho(v)=\overrightarrow{vw}$. Then $w$ is also on the cycle. Let $T$ be the spanning tree we get by forgetting the orientations of $\{\varrho(u): u\in V(G)-v \}$. Take $(\mathbf{0}_G, \varrho)$, and route at $w$ until the rotor at $w$ becomes $\overrightarrow{wv}$. Let the DRC at this moment be $(x,\varrho')$. Then $(x,\varrho')\sim (\mathbf{0}_G,\varrho)$ by its construction.
Let $T'$ be the subgraph that we get by forgetting the orientations of $\{\varrho'(u) :u\in V(G)-w\}=\{\varrho(u) :u\in V(G)-w\}$. $T'$ is a spanning tree, since $\varrho$ has one cycle, and $w$ is on this cycle.
Note that $T'_v\cup \overrightarrow{vw}=T'_w\cup \overrightarrow{wv}$.

As $(x,T'_v\cup \overrightarrow{vw})=(x,\varrho')\sim (\mathbf{0}_G,\varrho)=(\mathbf{0}_G,T_v\cup \overrightarrow{vw})$, $x_v(T')=T$.
On the other hand, if $x_w(T')=T$ were true, that would mean, using also the previous equivalence, that  
$(\mathbf{0}_G,\varrho)\sim(x,T'_v\cup \overrightarrow{vw})=(x,T'_w\cup \overrightarrow{wv})\sim (\mathbf{0}_G, T_w \cup \overrightarrow{wv})=(\mathbf{0}_G,\overleftarrow{\varrho})$, contradicting our assumption.
\end{proof}

\section*{Acknowledgement}
Research was supported by the MTA-ELTE Egerv\'ary Research Group and by the
Hungarian Scientific Research Fund - OTKA, K109240.

We would like to thank Viktor Kiss for his helpful comments about the manuscript.

\bibliographystyle{abbrv}
\bibliography{rotor_r}

\end{document}